\documentclass[reqno]{amsart}

\allowdisplaybreaks

\newtheorem{theorem}{Theorem}[section]
\newtheorem{lemma}[theorem]{Lemma}
\newtheorem{proposition}[theorem]{Proposition}

\usepackage{amssymb,amstext,latexsym}
\usepackage{nicefrac}

\usepackage{hyperref}

\usepackage{tikz}
\usetikzlibrary{decorations.pathreplacing}

\usepackage[nobysame,initials]{amsrefs}

\newcommand\R{\mathbb{R}}
\newcommand\B{\mathbb{B}}
\newcommand\CO{\mathcal{O}}

\DeclareMathOperator{\conv}{conv}

\title[Minimum-area circumscribed quadrangles]{On the minimal area of quadrangles circumscribed about planar
convex bodies}

\author{Ferenc Fodor} 
\address{Bolyai Institute, University of Szeged, Aradi v\'ertan\'uk tere 1, 6720 Szeged, Hungary} 
\email{fodorf@math.u-szeged.hu}

\author{Florian Grundbacher}
\address{Department of Mathematics, Technical University of Munich, Boltzmannstrasse 3, 85748 Garching, Germany}
\email{florian.grundbacher@tum.de}

\date{June 5, 2025}

\begin{document}

\begin{abstract}
We show that every planar convex body is contained in a quadrangle whose area is less than $(1 - 2.6 \cdot 10^{-7}) \sqrt{2}$ times the area of the original convex body, improving the best known upper bound by W. Kuperberg.
\end{abstract}

\maketitle

\section{Introduction and Results}

It is one of the classical problems in discrete geometry to investigate the approximation properties of planar convex figures by convex $n$-gons.
This problem is not only interesting for its own sake, but it also plays an important role in the theory of packing and covering, see L.~Fejes T\'oth \cite{Ft}, and in computational geometry, for references, see Hong et al.~\cite{H23}.

Let $K$ be a convex body (compact convex set with non-empty interior) in the Euclidean plane $\R^2$. For $n\geq 3$, let $C_n(K)$ ($I_n(K)$) be a convex polygon with at most $n$ sides containing $K$ (contained in $K$) that has minimal (maximal) area.
Let us denote the area of (Lebesgue) measurable sets by $|\cdot|$. The quantities 
\[
    c_n=\inf \frac{|C_n(K)|}{|K|},
\]
and
\[
    i_n=\sup \frac{|I_n(K)|}{|K|},
\]
where the infimum and the supremum are taken over all convex bodies $K$ in $\R^2$, have been studied extensively. It can be shown by standard compactness arguments that both $c_n$ and $i_n$ are, in fact, attained. In the inscribed case, Sas \cite{Sas} proved that 
\[
    i_n=\frac{n}{2\pi}\sin \frac{2\pi}{n},
\]
where equality holds if and only if $K$ is an ellipse. As is too common in approximation problems, we know less about the circumscribed case. The only $c_i$ whose exact value is known is $c_3=2$, and the extremal figures are parallelograms; see Gross \cite{Gr}, Eggleston \cite{E}, and Chakerian \cite{Ch}. The upper estimate $c_4\leq \sqrt 2$ was proved by Chakerian and Lange \cite{ChLa} (see also Chakerian \cite{Ch}). W.~Kuperberg \cite{Ku} proved that $c_4<\sqrt 2$, but did not give an explicit upper bound less than $\sqrt 2$. Kuperberg conjectured that the extremal figures are the affine regular pentagons and that $c_4= \frac{3}{\sqrt 5}$. Hong et al.~\cite{H23} proved that every unit-area convex
pentagon is contained in a convex quadrangle of area at most $\frac{3}{\sqrt 5}$, and that equality holds only for affine regular pentagons, thus $c_4\geq \frac{3}{\sqrt 5}$. 
However, $c_4$ is still unknown for general planar convex bodies $K$, and not even an explicit $\varepsilon>0$ has been found such that $c_4\leq (1-\varepsilon)\sqrt 2$ for all $K$.

We note that in the special case when $K$ is centrally symmetric, Petty \cite{P55} proved that the area of the minimum-area parallelogram circumscribed about $K$ is at most $\frac 43 |K|$, with equality only for affine regular hexagons. Since minimum-area circumscribed quadrangles around centrally symmetric planar convex bodies are also centrally symmetric with the same centre, $c_4=\frac 43$ for such planar convex bodies (see also Pe\l{}czi\'nski and Szarek \cite{PSz}). 

For $5\leq n\leq 11$, the currently known best upper bound is $c_n\leq \sec\frac\pi n$ due to Ismailescu \cite{Is}, who improved on an earlier bound by Chakerian \cite{Ch}. For $n\geq 12$, L.~Fejes T\'oth's bounds \cite{FTL40} are still the best: 
$$\frac n\pi\tan\frac\pi n\leq c_n\leq \frac {n-2}\pi\tan\frac\pi {n-2}.$$
For more information and references about related results, we suggest the reader consult the Notes section of L.~Fejes T\'oth's seminal book \cite{Ft}, the \emph{Lagerungen}, which has recently been translated into English by G.~Fejes T\'oth and W.~Kuperberg and which contains an up-to-date survey on the current state of the art on problems mentioned in the book. We also refer to the introduction of Hong et al.~\cite{H23} for an overview of results on approximations of convex $m$-gons by $n$-gons. 

Our main theorem adds to Kuperberg's result \cite{Ku} in which it is proved that $c_4<\sqrt2$. 
Our goal is to find an $\varepsilon>0$ such that for every planar convex body $K$ there exists a quadrangle $Q$ such that $K\subset Q$ and $|Q|\leq (1-\varepsilon)\sqrt 2 \, |K|$. We prove the following theorem. 

\begin{theorem} \label{thm:main}
Let $K \subset \R^2$ be a convex body.
Then there exists a quadrangle $Q \subset \R^2$ with $K \subset Q$ and $|Q| < (1 - 2.6 \cdot 10^{-7}) \sqrt{2} \, |K|$.
\end{theorem}

The methods we employ can be considered rather elementary,
and our ideas are based on an approach by Ismailescu \cite{Is}.
He also gave a proof of the fact that $c_4\leq \sqrt{2}$,
and it turns out that using appropriate stability improvements in his arguments
allows us to find an explicit $\varepsilon$ as described above.

In Section~\ref{sec:Ismailescu}, we describe Ismailescu's proof with some modifications. Section~\ref{sec:improvements} contains some geometric preparations. We prove Theorem~\ref{thm:main} in Section~\ref{sec:main}.

\section{Original Proof by Ismailescu} \label{sec:Ismailescu}

Let us first introduce the necessary notation.
For $X,Y \subset \R^2$, $v \in \R^2$, and $\rho \in \R$,
we write $X + Y := \{ x + y : x \in X, y \in Y\}$ for the Minkowski sum of $X$ and $Y$.
The $v$-translation and $\rho$-dilatation of $X$ are given by $v + X = \{v\} + X$ and $\rho X := \{ \rho x : x \in X\}$, respectively.
The convex hull of $X$ is abbreviated as $\conv(X)$.
We denote the maximum norm on $\R^2$ by $\left\| \cdot \right\|_\infty$, with associated unit ball $\B_\infty^2$.

Second, we discuss the existence of minimum-area quadrangles circumscribed about a planar convex body $K \subset \R^2$. Note that such a quadrangle always has to be convex.
Moreover, it is not difficult to see that there exist some points $v^1, \ldots, v^4 \in \R^2$ such that $K \subset Q := \conv(\{ v^1, \ldots, v^4 \})$ and $Q$ has minimum possible area.
If $Q$ has precisely three vertices, then all of them must belong to $K$, that is, $K = Q$.
Otherwise, we could separate $K$ from any vertex of $Q$ that $K$ does not contain
and intersect $Q$ with the separating half-space containing $K$
to obtain another convex hull of at most four points with smaller area than $Q$.
If $Q$ has at most two vertices, then $|Q| = |K| = 0$, which cannot happen since $K$ has non-empty interior.
Either way, the factor $\sqrt{2}$ can be reduced to $1$.
Throughout the paper, we may therefore always assume that $Q$ is a proper convex quadrangle, that is, $Q$ has exactly four vertices.

Next, we require the following two results.
The first can be found, for example, in \cite{Ft}*{p.~ 9}.

\begin{proposition} \label{prop:min_quad}
Let $K \subset \R^2$ be a convex body and $Q \subset \R^2$ a proper quadrangle
such that $Q$ is a minimum-area circumscribed quadrangle of $K$.
Then all midpoints of the edges of $Q$ belong to $K$.
\end{proposition}

The second result is known as Varignon's theorem.

\begin{proposition} \label{prop:midpoints}
Let $Q$ be a proper convex quadrangle and $P \subset Q$ be a convex quadrangle with vertices at the midpoints of the edges of $Q$.
Then $P$ is a parallelogram, and $|Q| = 2 \, |P|$.
\end{proposition}

We now turn to the proof by Ismailescu.

\begin{proposition} \label{prop:ismailescu}
Let $K \subset \R^2$ be a convex body.
Then there exists a quadrangle $Q \subset \R^2$ with $K \subset Q$ and $|Q| \leq \sqrt{2} \, |K|$.
\end{proposition}
\begin{proof}[Proof by Ismailescu \cite{Is}]
We may assume that there exists a proper convex quadrangle $Q$ such that $K \subset Q$
and $Q$ has minimum area among all circumscribed quadrangles of $K$.
Let $P$ be the convex quadrangle formed from the midpoints of the edges of $Q$.
By Proposition~\ref{prop:min_quad}, we have $P \subset K$.
Moreover, Proposition~\ref{prop:midpoints} shows that $P$ is a parallelogram with $|Q| = 2 \, |P|$.
By applying an appropriate affine transformation if necessary, we may assume that $P = \B_\infty^2=\conv(\{(\pm 1,0), (0,\pm 1)\})$.

Now, let $a_1, a_2 \leq -1$ and $b_1, b_2 \geq 1$ such that $Q' = \{ x \in \R^2 : a_1 \leq x_1 \leq b_1, a_2 \leq x_2 \leq b_2 \}$
is an axis-parallel quadrangle circumscribed about $K$.
In particular, there exist $v^1, v^2, w^1, w^2 \in K$ such that $v^1_1 = a_1$, $v^2_2 = a_2$, $w^1_1 = b_1$, and $w^2_2 = b_2$.
Let $\CO = \conv (P \cup \{ v^1, v^2, w^1, w^2 \})$ be the octagon shown in Figure~\ref{fig:Ismailescu}.
Since $K$ is convex and the vertices of $P$ are on the boundary of $K$,
we can compute $|\CO|$ by splitting it into four quadrangles
whose vertices are the origin $0$, two adjacent vertices of $P$, and one of $v^1, v^2, w^1, w^2$, respectively, such that
\[
    |\CO|
    = \frac{1}{2} (2 |a_1|) + \frac{1}{2} (2 |a_2|) + \frac{1}{2} (2 b_1) + \frac{1}{2} (2 b_2)
    = (b_1 - a_1) + (b_2 - a_2).
\]
Writing $x = b_1 - a_1$ and $y = b_2 - a_2$, we obtain $|\CO| = x + y$.
Moreover, we note that $|Q'| = x y$.
By assumption and the above, we have $8 = 2 \, |P| = |Q| \leq |Q'|$.
Thus,
\begin{align}
\begin{split} \label{eq:main_estimate}
    |Q|
    & \leq \sqrt{ |Q| \, |Q'| }
    = \sqrt{ 2 \, |P| \, |Q'| }
    \\
    & = \sqrt{2} \, \sqrt{ 4 xy }
    \leq \sqrt{2} \, \sqrt{ x^2 + 2 xy + y^2 }
    = \sqrt{2} \, |\CO|
    \leq \sqrt{2} \, |K|,
\end{split}
\end{align}
where we used that $4 xy \leq x^2 + 2 xy + y^2$ by $x^2 - 2 xy + y^2 = (x-y)^2 \geq 0$.
\end{proof}

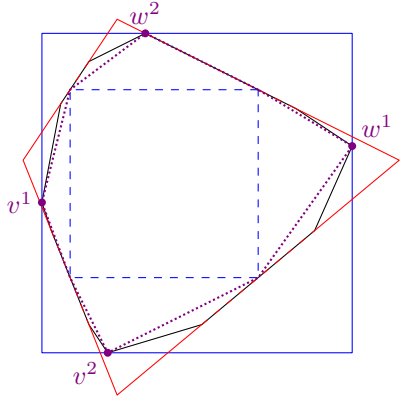
\begin{figure}[ht]
\newcommand\maxx{2}
\newcommand\maxxy{2/5}
\newcommand\minx{-13/10}
\newcommand\minxy{-1/5}
\newcommand\maxy{8/5}
\newcommand\maxyx{-1/5}
\newcommand\miny{-9/5}
\newcommand\minyx{-3/5}
\newcommand\scale{1.25}
\newcommand\ds{1.5/\scale pt}
\begin{tikzpicture}[line join=round, scale=\scale]
\draw[blue,dashed] (1,1) -- (-1,1) -- (-1,-1) -- (1,-1) -- cycle;

\draw (7/5,4/5) -- (-1/5,8/5) -- (-4/5,13/10) -- (-11/10,17/20) -- (-13/10,-1/5) -- (-6/5,-1/2) -- (-4/5,-3/2) -- (-3/5,-9/5) -- (2/5,-3/2) -- (8/5,-1/2) -- (2,2/5) -- cycle;

\draw[blue] ({\maxx},{\maxy}) -- ({\minx},{\maxy}) -- ({\minx},{\miny}) -- ({\maxx},{\miny}) -- cycle;

\draw[red] (8/5,-1/2) -- (5/2,1/4) -- (7/5,4/5);
\draw[red,dashed] (7/5,4/5) -- (-1/5,8/5);
\draw[red] (-1/5,8/5) -- (-1/2,7/4) -- (-4/5,13/10);
\draw[red,dashed] (-4/5,13/10) -- (-11/10,17/20);
\draw[red] (-11/10,17/20) -- (-3/2,1/4) -- (-6/5,-1/2);
\draw[red,dashed] (-6/5,-1/2) -- (-4/5,-3/2);
\draw[red] (-4/5,-3/2) -- (-1/2,-9/4) -- (2/5,-3/2);
\draw[red,dashed] (2/5,-3/2) -- (8/5,-1/2);

\draw[violet,densely dotted, thick] (1,1) -- ({\maxyx},{\maxy}) -- (-1,1) -- ({\minx},{\minxy}) -- (-1,-1) -- ({\minyx},{\miny}) -- (1,-1) -- ({\maxx},{\maxxy}) -- cycle;

\fill[violet] ({\minx},{\minxy}) circle(\ds) node[anchor=east] {$v^1$};
\fill[violet] ({\minyx},{\miny}) circle(\ds) node[anchor=north east] {$v^2$};
\fill[violet] ({\maxx},{\maxxy}) circle(\ds) node[anchor=south west] {$w^1$};
\fill[violet] ({\maxyx},{\maxy}) circle(\ds) node[anchor=south] {$w^2$};
\end{tikzpicture}
\caption{An example for the situation in the proof of Proposition~\ref{prop:ismailescu}: $K$ (black), $Q$ (red), $Q'$ (blue, solid), $P = \B_\infty^2$ (blue, dashed), $\CO$ (violet).}
\label{fig:Ismailescu}
\end{figure}

The estimate $\sqrt{ |P| \, |Q'| } \leq |\CO|$ was obtained through the Brunn--Minkowski inequality
in Ismailescu's (more general) original proof.
For our purposes, it is more straightforward to obtain it by elementary methods,
so that we do not need to consider stability estimates for the Brunn--Minkowski inequality later on.

Our slight improvement of the factor $\sqrt{2}$ is now based on
the idea of obtaining ``stability improvements'' for the steps in the main estimation in (the above version of) Ismailescu's proof.
In essence, we show that the factor can be improved trivially or $Q'$ must be almost a square of area $8$.
In the latter case, we show that there exists a quadrangle $\tilde{Q}$ that contains $\CO$
and has area $\gamma \, |\CO|$ for some explicit $\gamma < \sqrt{2}$.
Finally, we can again either trivially improve the factor $\sqrt{2}$ for $K$
or $K$ is contained in an only slightly larger dilatation of $\CO$,
where the previous fact gives an improvement of the general factor even in the last case.

\section{Improvement for Special Cases of Octagons} \label{sec:improvements}

In the following, we shall verify the claimed existence of a quadrangle $\tilde{Q}$
with $\CO \subset \tilde{Q}$ and $|\tilde{Q}| \leq \gamma |\CO|$
in case $Q'$ is almost a square of area $8$.
To simplify the statement of the result, we define for $c \geq \frac{14}{5}$ and $\delta \in [0,\frac{1}{10}]$ the function
\begin{align*}
    \zeta_{c,\delta}(t)
    & = \frac{c}{5 (1 - 2 \delta)}
    \cdot \frac{(c (3 - 4 \delta) + t - 1) (4 c (3 - 4 \delta) + (7 + 4 \delta - 20 \delta^2) (t-1))}{c (9 - 20 \delta + 20 \delta^2) + 4 (t - 1)}
    \\
    & \qquad\qquad - \frac{c}{5 (1 - 2 \delta)} (3 - 4 \delta) (t-1).
\end{align*}
This function is differentiable for all $t > -\frac{c}{2}$ with derivative
\[
    - \frac{2 c (1 - 2 \delta) (c^2 (9 - 48 \delta + 108 \delta^2 - 80 \delta^3) + c (9 - 20 \delta + 20 \delta^2) (t - 1) + 2 (t - 1)^2)}
    {(c (9 - 20 d + 20 d^2) + 4 (t - 1))^2}.
\]
By $c > 0$ and $\delta \leq \frac{1}{10}$, the sign of $\zeta_{d,\delta}'(t)$ is the same as the sign of
\begin{align*}
    & -\left( c^2 (9 - 48 \delta + 108 \delta^2 - 80 \delta^3) + c (9 - 20 \delta + 20 \delta^2) (t - 1) + 2 (t - 1)^2 \right)
    \\
    & = -2 \left( t - ( 1 - c (3 - 4\delta) ) \right) \left( t - \left( 1 - c \frac{3 - 12 \delta + 20 \delta^2}{2} \right) \right).
\end{align*}
Now, note that $c \geq \frac{14}{5}$ and $\delta \leq \frac{1}{10}$ give
\[
    1 - c (3 - 4 \delta)
    < 1 - c \frac{3 - 12 \delta + 20 \delta^2}{2}
    \leq 1 - c
    < - \frac{c}{2}.
\]
Therefore, we have $\zeta_{c,\delta}'(t) < 0$, meaning that $\zeta_{c,\delta}$ is decreasing on $[-\frac{c}{2},\infty)$.
In particular, we obtain
\begin{align}
\begin{split} \label{eq:zeta_dec}
    \zeta_{c,\delta}(t)
    & \leq \zeta_{c,\delta} \left( -\frac{c}{2} \right)
    \\
    & = \frac{c}{20} \left( 8 (c + 2) + \frac{4 c}{1 - 2 \delta} + \frac{(c - 2) (c (43 - 54 \delta) - 22 - 20 \delta)}{c (7 - 20 \delta + 20 \delta^2) - 4} \right).
\end{split}
\end{align}

\begin{lemma} \label{lem:octagon}
Let $v^1, v^2, w^1, w^2 \in \R^2$ such that
\[
    v^1_2, v^2_1, w^1_2, w^2_1 \in [-1,1],
	\quad
    - v^1_1, - v^2_2, w^1_1, w^2_2 \geq 1,
	\quad \text{and} \quad
    w_1^1 - v^1_1, w^2_2 - v^2_2 \leq c
\]
for some $c \geq \frac{14}{5}$.
Let $\delta \in [0,\frac{1}{10}]$.
Then there exists a quadrangle $\tilde{Q} \subset \R^2$ with $\{ v^1, v^2, w^1, w^2 \} \cup \B_\infty^2 \subset \tilde{Q}$ and
\[
    |\tilde{Q}|
    \leq \max \left\{
        \frac{c (c (1+\delta) + 2\delta))}{1 + 2 \delta},
	\zeta_{c,\delta} \left( -\frac{c}{2} \right)
    \right\}.
\]
\end{lemma}
\begin{proof}
We may assume that $-v^1_1 \leq w^1_1$ and $-v^2_2 \leq w^2_2$,
as we could otherwise reflect at the coordinate axes to ensure these assumptions.
In particular, we obtain $v^1_1, v^2_2 \geq - \frac{c}{2}$.

Our first step is based on the intuition (from Proposition~\ref{prop:min_quad})
that if the points $w^1$ and $w^2$ are not close to the midpoints of their respective edges of the square
$\conv(\{ (v^1_1,v^2_2), (v^1_1+c,v^2_2), (v^1_1, v^2_2+c), (v^1_1+c,v^2_2+c) \})$,
then we can find a quadrangle containing all of $S := \{ v^1, v^2, w^1, w^2 \} \cup \B_\infty^2$
with significantly enough smaller area then this square.
The square, which corresponds to (a superset of) $Q'$ from before,
serves as a starting point for a quadrangle that contains $S$,
where we want to slightly decrease the area using appropriate modifications.

Let us assume that $w^1_2 < v^2_2 + (\frac{1}{2} - \delta) c$.
We define a line $U \subset \R^2$ by
\[
    U
    = \left\{ (x_1,x_2) \in \R^2 : x_2 = - \frac{ (\nicefrac{1}{2} + \delta) c}{v^1_1 + c - 1} (x_1-1) + v^2_2 + c \right\}
\]
(see Figure~\ref{fig:case1}).
Then the points $(1, v_2^2+c)$ and $(v^1_1+c, v_2^2 + (\frac{1}{2}-\delta) c)$ lie on $U$.
Note that $U$ has a negative slope since $v^1_1 + c - 1 \geq \frac{c}{2} - 1 > 0$.
Thus, from $v^2_2+c \geq \frac{c}{2} > 1$, we see that $v^1$, $w^1$, $w^2$, and all of $\B_\infty^2$ lie below $U$.
Moreover, $w^1_2 < v^2_2 + (\frac{1}{2} - \delta) c$ implies that $w^1$ is also below $U$.
Altogether, the half-space below $U$ contains all of $S$.

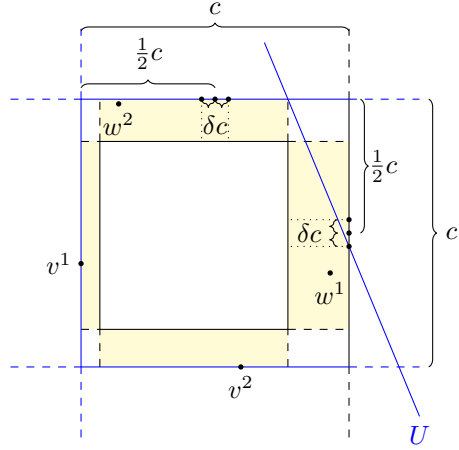
\begin{figure}[ht]
\newcommand\va{-1.2}
\newcommand\vb{-1.4}
\newcommand\vc{2.85}
\newcommand\vd{0.05}
\newcommand\pad{0.75}
\newcommand\scale{1.25}
\newcommand\ds{1/\scale pt}
\begin{tikzpicture}[line join=round, scale=\scale]
\fill[yellow!20] (\va,-1) -- (-1,-1) -- (-1,1) -- (\va,1) -- cycle;
\fill[yellow!20] (-1,\vb) -- (1,\vb) -- (1,-1) -- (-1,-1) -- cycle;
\fill[yellow!20] (1,-1) -- (\va+\vc,-1) -- (\va+\vc,1) -- (1,1) -- cycle;
\fill[yellow!20] (-1,1) -- (1,1) -- (1,\vb+\vc) -- (-1,\vb+\vc) -- cycle;

\draw[dashed, blue] (\va,\vb-\pad) -- (\va,\vb);
\draw[dashed, blue] (\va,\vb+\vc) -- (\va,\vb+\vc+\pad);
\draw[dashed] (\va+\vc,\vb-\pad) -- (\va+\vc,\vb);
\draw[dashed] (\va+\vc,\vb+\vc) -- (\va+\vc,\vb+\vc+\pad);
\draw[dashed, blue] (\va-\pad,\vb) -- (\va,\vb);
\draw[dashed, blue] (\va+\vc,\vb) -- (\va+\vc+\pad,\vb);
\draw[dashed, blue] (\va-\pad,\vb+\vc) -- (\va,\vb+\vc);
\draw[dashed, blue] (\va+\vc,\vb+\vc) -- (\va+\vc+\pad,\vb+\vc);

\draw[dashed] (\va,-1) -- (-1,-1);
\draw[dashed] (1,-1) -- (\va+\vc,-1);
\draw[dashed] (\va,1) -- (-1,1);
\draw[dashed] (1,1) -- (\va+\vc,1);
\draw[dashed] (-1,\vb) -- (-1,-1);
\draw[dashed] (-1,1) -- (-1,\vb+\vc);
\draw[dashed] (1,\vb) -- (1,-1);
\draw[dashed] (1,1) -- (1,\vb+\vc);

\draw[dotted] (\va+\vc,{\vb+(1/2+\vd)*\vc}) -- (1,{\vb+(1/2+\vd)*\vc});
\draw[dotted] (\va+\vc,{\vb+(1/2-\vd)*\vc}) -- (1,{\vb+(1/2-\vd)*\vc});

\draw[dotted] ({\va+(1/2+\vd)*\vc},\vb+\vc) -- ({\va+(1/2+\vd)*\vc},1);
\draw[dotted] ({\va+(1/2-\vd)*\vc},\vb+\vc) -- ({\va+(1/2-\vd)*\vc},1);

\draw (1,1) -- (-1,1) -- (-1,-1) -- (1,-1) -- cycle;

\draw[blue] (\va+\vc,\vb+\vc) -- (\va,\vb+\vc) -- (\va,\vb) -- (\va+\vc,\vb);
\draw (\va+\vc,\vb) -- (\va+\vc,\vb+\vc);

\draw[blue] plot[variable=\x,domain=0.75:\va+\vc+0.75] (\x,{-(0.5+\vd)*\vc/(\va+\vc-1)*(\x-1) + \vb+\vc}) node[below] {$U$};

\draw[decorate,decoration={brace,mirror,amplitude=3pt,raise=1ex}]
	(\va+\vc,{\vb+(1/2)*\vc}) -- (\va+\vc,{\vb+(1/2-\vd)*\vc});
\draw[decorate,decoration={brace,mirror,amplitude=3pt,raise=1ex}]
	(\va+\vc,{\vb+(1/2+\vd)*\vc}) -- (\va+\vc,{\vb+(1/2)*\vc});
\node[xshift=-3.5ex] at (\va+\vc,{\vb+(1/2)*\vc}) {$\delta c$};
\draw[decorate,decoration={brace,amplitude=3pt,raise=1ex}]
	(\va+\vc,\vb+\vc) -- (\va+\vc,{\vb+(1/2)*\vc}) node[midway,xshift=3ex] {$\frac{1}{2} c$};
\draw[decorate,decoration={brace,amplitude=3pt,raise=7ex}]
	(\va+\vc,\vb+\vc) -- (\va+\vc,\vb) node[midway,xshift=9ex] {$c$};

\draw[decorate,decoration={brace,mirror,amplitude=3pt,raise=1/2ex}]
	({\va+(1/2-\vd)*\vc},\vb+\vc) -- ({\va+(1/2)*\vc},\vb+\vc);
\draw[decorate,decoration={brace,mirror,amplitude=3pt,raise=1/2ex}]
	({\va+(1/2)*\vc},\vb+\vc) -- ({\va+(1/2+\vd)*\vc},\vb+\vc);
\node[yshift=-2.25ex] at ({\va+(1/2)*\vc},\vb+\vc) {$\delta c$};
\draw[decorate,decoration={brace,amplitude=3pt,raise=1ex}]
	(\va,\vb+\vc) -- ({\va+(1/2)*\vc},\vb+\vc) node[midway,yshift=4ex] {$\frac{1}{2} c$};
\draw[decorate,decoration={brace,amplitude=3pt,raise=6ex}]
	(\va,\vb+\vc) -- (\va+\vc,\vb+\vc) node[midway,yshift=8ex] {$c$};

\fill (\va,-0.3) circle(\ds) node[anchor=east] {$v^1$};
\fill (0.5,\vb) circle(\ds) node[anchor=north] {$v^2$};
\fill (\va+\vc-0.2,-0.4) circle(\ds) node[anchor=north] {$w^1$};
\fill (-0.8,\vb+\vc-0.05) circle(\ds) node[anchor=north] {$w^2$};

\fill (\va+\vc,{\vb+(1/2)*\vc}) circle(\ds);
\fill (\va+\vc,{\vb+(1/2+\vd)*\vc}) circle(\ds);
\fill (\va+\vc,{\vb+(1/2-\vd)*\vc}) circle(\ds);

\fill ({\va+(1/2)*\vc},\vb+\vc) circle(\ds);
\fill ({\va+(1/2+\vd)*\vc},\vb+\vc) circle(\ds);
\fill ({\va+(1/2-\vd)*\vc},\vb+\vc) circle(\ds);
\end{tikzpicture}
\caption{An example for the situation in the first part of the proof of Lemma~\ref{lem:octagon}.
The points $v^1$, $v^2$, $w^1$, and $w^2$ lie somewhere in the yellow areas surrounding the inner square $\B_\infty^2$.
If one of $w^1$ and $w^2$ does not lie between the dotted segments in their respective areas,
then four lines like the blue ones here can be used to obtain a quadrangle $\tilde{Q}^1$
with significantly enough smaller area than the outer square such that it still contains all of $S$.}
\label{fig:case1}
\end{figure}

Next, note that the unique point in $U$ with $y$-coordinate equal to $v^2_2$ is $(1+\frac{v^1_1+c-1}{\nicefrac{1}{2} + \delta}, v^2_2)$.
Since any $x \in S$ satisfies $v^1_1 \leq x_1$ and $v^2_2 \leq x_2 \leq w^2_2 \leq v^2_2 + c$,
we obtain that all of $S$ is contained in the quadrangle
\[
    \tilde{Q}^1
    := \conv \left( \left\{
        (1,v^2_2+c),
        (v^1_1,v^2_2+c),
        (v^1_1,v^2_2),
        \left( 1+\frac{v^1_1+c-1}{\nicefrac{1}{2} + \delta}, v^2_2 \right)
    \right\} \right).
\]
We can easily compute the area of $\tilde{Q}^1$ to be
\[
    (1-v^1_1) c + \frac{1}{2} \cdot \frac{v^1_1+c-1}{\nicefrac{1}{2} + \delta} c
    = \frac{c (c + 2 \delta - 2 v^1_1 \delta)}{1 + 2 \delta}.
\]
Since $v^1_1 \geq -\frac{c}{2}$, we therefore obtain
\[
    |\tilde{Q}^1|
    \leq \frac{c (c (1+\delta) + 2\delta))}{1 + 2 \delta}.
\]
Altogether, the claim holds in this case.
Using analogous constructions for the three cases where $w^1_2 > v^2_2 + (\frac{1}{2} + \delta) c$,
$w^2_1 < v^1_1 + (\frac{1}{2} - \delta) c$, or $w^2_1 > v^1_1 + (\frac{1}{2} + \delta) c$,
we may from now on assume that $|w^1_2 - (v^2_2 + \frac{c}{2})| \leq \delta c$ and $|w^2_1 - (v^1_1 + \frac{c}{2})| \leq \delta c$.

For the second step, we use the additional restrictions on $w^1$ and $w^2$ to find further quadrangles that contain all of $S$.
We define two new lines by
\[
    U'
    := \left\{ (x_1,x_2) \in \R^2 : x_2 = \frac{(\nicefrac{1}{2} - \delta) c}{v^1_1 + c - 1} (x_1 - 1) + v^2_2 \right\}
\]
and
\[
    U''
    := \left\{ (x_1,x_2) \in \R^2 : x_2 = - \frac{1}{5 (\nicefrac{1}{2} - \delta)} (x_1 - (v^1_1+c)) + v^2_2 + \frac{4}{5} c \right\}
\]
(see Figure~\ref{fig:case2}).

\begin{figure}[ht]
\newcommand\va{-1.2}
\newcommand\vb{-1.4}
\newcommand\vc{2.85}
\newcommand\vd{0.05}
\newcommand\pad{0.75}
\newcommand\scale{1.25}
\newcommand\ds{1/\scale pt}
\begin{tikzpicture}[line join=round, scale=\scale]
\fill[yellow!20] (\va,-1) -- (-1,-1) -- (-1,1) -- (\va,1) -- cycle;
\fill[yellow!20] (-1,\vb) -- (1,\vb) -- (1,-1) -- (-1,-1) -- cycle;
\fill[yellow!20] (1,-1) -- (\va+\vc,-1) -- (\va+\vc,1) -- (1,1) -- cycle;
\fill[yellow!20] (-1,1) -- (1,1) -- (1,\vb+\vc) -- (-1,\vb+\vc) -- cycle;

\draw[dashed, blue] (\va,\vb-\pad) -- (\va,\vb);
\draw[dashed, blue] (\va,\vb+\vc) -- (\va,\vb+\vc+\pad);
\draw[dashed] (\va+\vc,\vb-\pad) -- (\va+\vc,\vb);
\draw[dashed] (\va+\vc,\vb+\vc) -- (\va+\vc,\vb+\vc+\pad);
\draw[dashed, blue] (\va-\pad,\vb) -- (\va,\vb);
\draw[dashed, blue] (\va+\vc,\vb) -- (\va+\vc+\pad,\vb);
\draw[dashed] (\va-\pad,\vb+\vc) -- (\va,\vb+\vc);
\draw[dashed] (\va+\vc,\vb+\vc) -- (\va+\vc+\pad,\vb+\vc);

\draw[dashed] (\va,-1) -- (-1,-1);
\draw[dashed] (1,-1) -- (\va+\vc,-1);
\draw[dashed] (\va,1) -- (-1,1);
\draw[dashed] (1,1) -- (\va+\vc,1);
\draw[dashed] (-1,\vb) -- (-1,-1);
\draw[dashed] (-1,1) -- (-1,\vb+\vc);
\draw[dashed] (1,\vb) -- (1,-1);
\draw[dashed] (1,1) -- (1,\vb+\vc);

\draw[dotted] (\va+\vc,{\vb+(1/2+\vd)*\vc}) -- (1,{\vb+(1/2+\vd)*\vc});
\draw[dotted] (\va+\vc,{\vb+(1/2-\vd)*\vc}) -- (1,{\vb+(1/2-\vd)*\vc});

\draw[dotted] ({\va+(1/2+\vd)*\vc},\vb+\vc) -- ({\va+(1/2+\vd)*\vc},1);
\draw[dotted] ({\va+(1/2-\vd)*\vc},\vb+\vc) -- ({\va+(1/2-\vd)*\vc},1);

\draw (1,1) -- (-1,1) -- (-1,-1) -- (1,-1) -- cycle;

\draw[blue] (\va,\vb+\vc) -- (\va,\vb) -- (\va+\vc,\vb);
\draw (\va+\vc,\vb) -- (\va+\vc,\vb+\vc) -- (\va,\vb+\vc);

\draw[blue] plot[variable=\x,domain=0.75:\va+\vc+1] (\x,{(0.5-\vd)*\vc/(\va+\vc-1)*(\x-1) + \vb}) node[right] {$U'$};
\draw[blue] plot[variable=\x,domain=-1.5:\va+\vc+1] (\x,{-1/(5*(0.5-\vd))*(\x-(\va+\vc)) + \vb + 4/5*\vc}) node[right] {$U''$};

\draw[decorate,decoration={brace,amplitude=3pt,raise=1ex}]
	(\va+\vc,{\vb+(1/2+\vd)*\vc}) -- (\va+\vc,{\vb+(1/2)*\vc}) node[midway,xshift=3ex] {$\delta c$};
\draw[decorate,decoration={brace,amplitude=3pt,raise=1ex}]
	(\va+\vc,\vb+\vc) -- (\va+\vc,{\vb+(4/5)*\vc}) node[midway,xshift=3ex] {$\frac{1}{5} c$};
\draw[decorate,decoration={brace,amplitude=3pt,raise=1ex}]
	(\va+\vc,{\vb+(1/2)*\vc}) -- (\va+\vc,\vb) node[midway,xshift=3ex] {$\frac{1}{2} c$};
\draw[decorate,decoration={brace,amplitude=3pt,raise=7ex}]
	(\va+\vc,\vb+\vc) -- (\va+\vc,\vb) node[midway,xshift=9ex] {$c$};

\draw[decorate,decoration={brace,amplitude=3pt,raise=1ex}]
	({\va+(1/2-\vd)*\vc},\vb+\vc) -- ({\va+(1/2)*\vc},\vb+\vc) node[midway,yshift=4ex] {$\delta c$};
\draw[decorate,decoration={brace,amplitude=3pt,raise=1ex}]
	({\va+(1/2)*\vc},\vb+\vc) -- (\va+\vc,\vb+\vc) node[midway,yshift=4ex] {$\frac{1}{2} c$};
\draw[decorate,decoration={brace,amplitude=3pt,raise=6ex}]
	(\va,\vb+\vc) -- (\va+\vc,\vb+\vc) node[midway,yshift=8ex] {$c$};

\fill (\va,-0.3) circle(\ds) node[anchor=east] {$v^1$};
\fill (0.5,\vb) circle(\ds) node[anchor=north] {$v^2$};
\fill (\va+\vc-0.15,0.1) circle(\ds) node[anchor=east] {$w^1$};
\fill (0.15,\vb+\vc-0.05) circle(\ds) node[anchor=north] {$w^2$};

\fill (\va+\vc,{\vb+(1/2)*\vc}) circle(\ds);
\fill (\va+\vc,{\vb+(1/2+\vd)*\vc}) circle(\ds);
\fill (\va+\vc,{\vb+(1/2-\vd)*\vc}) circle(\ds);
\fill (\va+\vc,{\vb+(4/5)*\vc}) circle(\ds);

\fill ({\va+(1/2)*\vc},\vb+\vc) circle(\ds);
\fill ({\va+(1/2+\vd)*\vc},\vb+\vc) circle(\ds);
\fill ({\va+(1/2-\vd)*\vc},\vb+\vc) circle(\ds);
\end{tikzpicture}
\caption{An example for the situation in the second part of the proof of Lemma~\ref{lem:octagon}.
Since $w^1$ and $w^2$ are assumed to lie between the dotted segments within their respective yellow areas,
we know that all of $S$ lies in the quadrangle $\tilde{Q}^2$ bounded by the four blue lines.}
\label{fig:case2}
\end{figure}
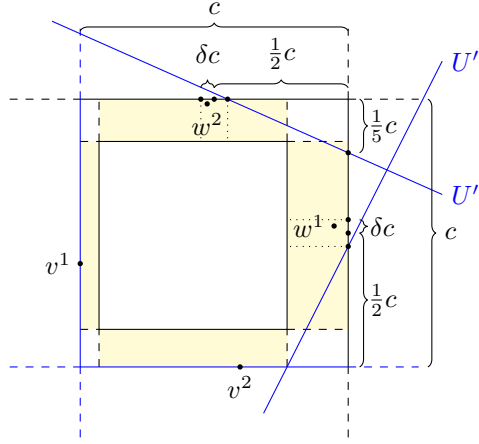

Clearly, $(1,v^2_2), (v^1_1+c,v^2_2 + (\frac{1}{2} - \delta) c) \in U'$
and $(v^1_1 + (\frac{1}{2}+\delta) c, v^2_2+c), (v^1_1+c,v^2_2+\frac{4}{5} c) \in U''$.
Since $U'$ has positive slope, it follows that all of $v^1$, $v^2$, $w^2$, and $\B_\infty^2$ lie above $U'$.
Moreover, the assumption $w^1_2 \geq v^1_2 + (\frac{1}{2} - \delta) c$ gives that $w^1$ lies above $U'$ as well.
Next, we note that the assumptions $v^1_1, v^2_2 \geq - \frac{c}{2}$, $\delta \in [0,\frac{1}{10}]$, and $c \geq \frac{14}{5}$ yield
\[
    - \frac{1}{5 (\nicefrac{1}{2} - \delta)} (1 - (v^1_1+c)) + v^2_2 + \frac{4}{5} c
    \geq - \frac{1 - \nicefrac{c}{2}}{5 (\nicefrac{1}{2} - \delta)} - \frac{c}{2} + \frac{4}{5} c
    \geq - \frac{2-c}{5} + \frac{3c}{10}
    = \frac{c}{2} - \frac{2}{5}
    \geq 1.
\]
Since $U''$ has negative slope, it follows that $v^1$, $v^2$, and $\B_\infty^2$ lie below $U''$.
Moreover, the assumptions $w^2_1 \leq v^1_1 + (\frac{1}{2} + \delta) c$
and $w^1_2 \leq v^2_2 + (\frac{1}{2} + \delta) c \leq v^2_2 + \frac{4}{5} c$
yield that $w^1$ and $w^2$ lie below $U''$ as well.
Altogether, if we write $u^1$ for the unique intersection point of $U'$ and $U''$,
and $u^2$ for the unique point in $U''$ with $u^2_1 = v^1_1$, then
\[
    S
    \subset \tilde{Q}^2
    := \conv(\{ u^1, u^2, (v^1_1,v^2_2), (1,v^2_2) \}).
\]
Note that $u^2 = (v^1_1,\frac{c}{5 (\nicefrac{1}{2} - \delta)} + v^2_2 + \frac{4}{5} c)$
follows directly from the definition of $U''$.
Similarly, we obtain that $(1,h+v^2_2)$ for $h = \frac{c+v^1_1-1}{5 (\nicefrac{1}{2} - \delta)} + \frac{4}{5} c$
is the unique point in $U''$ with first coordinate $1$.
Additionally, we can find that
\[
    u^1_1
    = v_1^1 + \frac{c (4 c (3 - 4 \delta) + (7 + 4 \delta - 20 \delta^2) (v^1_1-1))}
    {c (9 - 20 \delta + 20 \delta^2) + 4 (v^1_1 - 1)}.
\]
Now, the area of $\tilde{Q}^2$ can be computed to be
\begin{align*}
    |\tilde{Q}^2|
    & = h (1-v^1_1) + \frac{1}{2} (1 - v^1_1) (u^2_2 - v^2_2 - h) + \frac{1}{2} h (u^1_1 - 1)
    \\
    & = \frac{1}{2} h (u^1_1 - v^1_1) + \frac{1}{2} (1-v^1_1) (u^2_2 - v^2_2)
    \\
    & = \frac{c (3 - 4 \delta) + v^1_1 - 1 }{5 (1 - 2 \delta)} \cdot \frac{c (4 c (3 - 4 \delta) + (7 + 4 \delta - 20 \delta^2) (v^1_1-1))}{c (9 - 20 \delta + 20 \delta^2) + 4 (v^1_1 - 1)}
    \\
    & \qquad\qquad - \frac{c (3 - 4 \delta) (v^1_1-1)}{5 (1 - 2 \delta)}
    = \zeta_{c,\delta}(v_1^1).
\end{align*}
Since $-\frac{c}{2} \leq v^1_1$, the claim now follows from \eqref{eq:zeta_dec}.
\end{proof}

\section{Proof of \texorpdfstring{Theorem~\ref{thm:main}}{Theorem 1.1}} \label{sec:main}

We verify two more preliminary lemmas before proving the main theorem.

\begin{lemma} \label{lem:outer_ball}
Let $Q \subset \R^2$ be a proper convex quadrangle such that the midpoints of its edges are the vertices of $\B_\infty^2$.
Then $Q \subset 3 \B_\infty^2$.
\end{lemma}
\begin{proof}
Assume, to the contrary,
that $Q \not\subset 3 \B_\infty^2$.
Then $Q$ must have a vertex $v \notin 3 \B_\infty^2$.
Since the vertices of $\B_\infty^2$ are boundary points of $Q$,
we see that $Q$ is contained in the union of the slabs
$\{ (x_1,x_2) \in \R^2 : -1 \leq x_1 \leq 1 \}$ and $\{ (x_1,x_2) \in \R^2 : -1 \leq x_2 \leq 1 \}$.
In particular, one coordinate of $v$ must have an absolute value of at most $1$.
Therefore, we may assume that $|v_1| \leq 1$ and $v_2 > 3$.

Now, let $u$ and $w$ be the vertices of $Q$ that are connected to $v$ via edges of $Q$.
We may assume
\[
    \frac{v+u}{2} = (1,1)
        \quad \text{and} \quad
    \frac{v+w}{2} = (-1,1).
\]
Then $u_2, w_2 < -1$ and $u_1, -w_1 \geq 1$, where we cannot have both $u_1 = 1$ and $w_1 = -1$.
But then one of $u$ and $w$ has two coordinates with an absolute value larger than $1$,
which is a contradiction.
\end{proof}

\begin{lemma} \label{lem:inner_ball}
Let $v \in \R^2$ with $v \in R \B_\infty^2$ for some $R > 0$.
Then for any $r \in [0,R+1]$, we have
\[
    \left( 1 - \frac{r}{R+1} \right) v + \frac{r}{R+1} \B_\infty^2
    \subset \conv(\{ v \} \cup \B_\infty^2) \cap (v + r \B_\infty^2).
\]
\end{lemma}
\begin{proof}
The triangle inequality yields for any $u \in \B_\infty^2$ that
\begin{align*}
    \left\| \left( 1 - \frac{r}{R+1} \right) v + \frac{r}{R+1} u - v \right\|_\infty
    & = \frac{r}{R+1} \left\| u - v \right\|_\infty
    \\
    & \leq \frac{r}{R+1} (\left\| u \right\|_\infty + \left\| v \right\|_\infty)
    \leq \frac{r}{R+1} (1+R)
    = r.
\end{align*}
This clearly suffices to verify the claim.
\end{proof}

We are now finally ready to verify the initially discussed improvement of the factor $\sqrt{2}$.

\begin{proof}[Proof of Theorem~\ref{thm:main}.]
We reuse the notation introduced in the proof of Proposition~\ref{prop:ismailescu}.
We shall distinguish a few cases where we introduce constants for each individual step.
These constants are chosen explicitly only at the end of the proof to allow us to improve the upper bound on the ratio of areas in the theorem.

First, assume that $|Q'| > c_1 |Q|$ for some $c_1 \geq 1$.
Then \eqref{eq:main_estimate} gives
$\sqrt{c_1} \, |Q| < \sqrt{ |Q| \, |Q'| } \leq \sqrt{2} \, |K|$ and in particular
\begin{equation} \label{eq:estimate1}
    |Q|
    < \frac{1}{\sqrt{c_1}} \sqrt{2} \, |K|.
\end{equation}
For the following considerations, we assume $|Q'| \leq c_1 |Q| = 8 c_1$.

Next, suppose that $x > c_2 y$ for some $c_2 \geq 1$.
Since clearly $y \geq 2$, we have
\[
    |\CO|^2 - |P| \, |Q'|
    = x^2 + 2 xy + y^2 - 4 xy
    = (x-y)^2
    > (c_2 - 1)^2 y^2
    \geq 4 (c_2 - 1)^2.
\]
From $|P| = 4$ and $|Q'| \leq 8 c_1$, it follows that
\begin{align*}
    |\CO|
    & > \sqrt{ |P| \, |Q'| + 4 (c_2 - 1)^2}
    \\
    & = \sqrt{ |P| \, |Q'| \left( 1 + \frac{4 (c_2 - 1)^2}{|P| \, |Q'|} \right) }
    \geq \sqrt{ |P| \, |Q'| \left( 1 + \frac{(c_2 - 1)^2}{8 c_1} \right) }.
\end{align*}
Thus, \eqref{eq:main_estimate} shows
\begin{equation} \label{eq:estimate2}
    |Q|
    \leq \sqrt{2 \, |P| \, |Q'|}
    < \frac{1}{\sqrt{1 + \nicefrac{(c_2-1)^2}{8 c_1}}} \sqrt{2} \, |\CO|
    \leq \frac{1}{\sqrt{1 + \nicefrac{(c_2-1)^2}{8 c_1}}} \sqrt{2} \, |K|.
\end{equation}
From now on, we assume that $x \leq c_2 y$.
Consequently, we have that
\[
    \frac{x^2}{c_2}
    \leq xy
    = |Q'|
    \leq 8 c_1
\]
and in particular $x \leq \sqrt{8 c_1 c_2}$.
Analogously, we may assume that $y \leq \sqrt{8 c_1 c_2}$.

Next, suppose that there exists some $v \in K$ such that $v + r \B_\infty^2$ does not meet $\CO$ for some $r > 0$.
Lemma~\ref{lem:outer_ball} shows that $Q \subset 3 \B_\infty^2$,
so $v \in K \subset Q$ yields in particular $v \in 3 \B_\infty^2$.
Moreover, $\B_\infty^2 = P \subset \CO$ is true as well, which implies $r \leq 2$.
Hence, Lemma~\ref{lem:inner_ball} shows that $K \supset \conv(\{v\} \cup \CO)$
contains a translated copy of $\frac{r}{4} \B_\infty^2$ that does not intersect $\CO$.
Therefore, $|\CO| \leq |Q| = 8$ yields
\[
    \left( 1 + \frac{r^2}{32} \right) |\CO|
    \leq |\CO| + \frac{r^2}{32} |Q|
    = |\CO| + \frac{r^2}{4}
    = |\CO| + \left| \frac{r}{4} \B_\infty^2 \right|
    < |K|,
\]
and by \eqref{eq:main_estimate} further
\begin{equation} \label{eq:estimate3}
    |Q|
    \leq \sqrt{2} \, |\CO|
    < \frac{1}{1 + \nicefrac{r^2}{32}} \sqrt{2} \, |K|.
\end{equation}
From now on, we assume that $v + r \B_\infty^2$ meets $\CO$ for all $v \in K$.
In particular, we have
\[
    K
    \subset \CO + r \B_\infty^2
    \subset (1+r) \CO.
\]

Finally, Lemma~\ref{lem:octagon} shows for any $c_3 \geq \sqrt{8 c_1 c_2} \geq x,y$ and $\delta \in [0,\frac{1}{10}]$
that there exists a quadrangle $\tilde{Q}$ with $\CO \subset \tilde{Q}$ and
\[
    |\tilde{Q}|
    \leq \max \left\{
        \frac{c_3 (c_3 (1 + \delta) + 2 \delta)}{1 + 2 \delta},
        \zeta_{c_3,\delta} \left( -\frac{c_3}{2} \right)
    \right\}.
\]
Since $K \subset (1+r) \tilde{Q}$ and $Q$ is a minimum-area quadrangle containing $K$ with $|Q| = 8$,
we must have
\begin{align*}
    8
    & \leq |(1+r) \tilde{Q}|
    = (1+r)^2 |\tilde{Q}|
    \\
    & \leq (1+r)^2 \max \left\{
        \frac{c_3 (c_3 (1 + \delta) + 2\delta)}{1 + 2 \delta},
        \zeta_{c_3,\delta} \left( - \frac{c_3}{2} \right)
    \right\}
    =: C(c_3,r,\delta).
\end{align*}
Consequently, if $c_1$, $c_2$, $c_3$, $r$, and $\delta$ are chosen such that
$\sqrt{8 c_1 c_2} \leq c_3$ and $C(c_3,r,\delta) < 8$,
then one of the above assumptions must be wrong,
that is, one of \eqref{eq:estimate1}, \eqref{eq:estimate2}, and \eqref{eq:estimate3} must be true.

We now explicitly choose
\begin{align*}
    c_1 & = 1 + 5.3 \cdot 10^{-7},
    \\
    c_2 & = 1 + \sqrt{8 c_1 (c_1 - 1)} < 1 + 2.06 \cdot 10^{-3}, \quad \text{and}
    \\
    r & = \sqrt{32 \left( \sqrt{c_1} - 1 \right)} < 2.913 \cdot 10^{-3},
\end{align*}
in which case the upper bounds in \eqref{eq:estimate1}, \eqref{eq:estimate2}, and \eqref{eq:estimate3} coincide.
We obtain $\sqrt{8 c_1 c_2} < 2.83134 =: c_3$
and choose $\delta = 2.824 \cdot 10^{-2}$.
In this case,
\[
    \frac{c_3 (c_3 + \delta (2+c_3))}{1 + 2 \delta}
    < 7.95359
\]
and with \eqref{eq:zeta_dec} also
\[
    \zeta_{c_3,\delta} \left( - \frac{c_3}{2} \right)
    < 7.95359.
\]
Finally,
\[
    C(c_3,r,\delta)
    < 7.95359 (1+r)^2
    < 7.999996.
\]
In summary, we have found an admissible choice of $c_1$, $c_2$, $c_3$, $r$, and $\delta$,
for which
\[
    \frac{1}{\sqrt{c_1}}
    = \frac{1}{\sqrt{1 + \nicefrac{(c_2-1)^2}{8 c_1}}}
    = \frac{1}{1 + \nicefrac{r^2}{32}}
    < 0.99999974
    = 1 - 2.6 \cdot 10^{-7}.
\]
This concludes the proof as discussed above.
\end{proof}

\section{Concluding Remarks}

We would like to point out the (absurdly) large gap between the conjectured value of $c_4$ by Kuperberg \cite{Ku} (which is the same as the established lower bound by Hong et al.~\cite{H23}) and the upper bound in Theorem~\ref{thm:main}. 
It is clear that there is room for improvement. We note that with careful tweaking of the argument and by distinguishing more cases, one could probably improve the upper bound by some orders of magnitude at the cost of making the paper less readable. However, the size of the gap remains. This reminds us of the packing density of the regular tetrahedron, which has long been known to be less than $1$ as regular tetrahedra do not tile space, see \cite{Ft}*{Section~16.8} for history of the problem and references. However, the best known upper bound is $(1-2.6\cdot 10^{-25})$  by Gravel, Elser, and Kallus \cite{GEK11}, while the best construction, by Chen, Engel, and Glotzer \cite{ChEG10}, yields $0.856347\ldots$, with a huge gap between them. Determining the exact value of $c_4$ is probably not as formidable as that of the packing density of the regular tetrahedron (for which there is not even an explicit conjecture) but being the "simplest" case of a long-standing approximation problem it certainly deserves further attention.

\section*{Acknowledgements}
The research of Ferenc Fodor was supported by project no.~150151, which has been implemented with the support provided by the Ministry of Culture and Innovation of Hungary from the National Research, Development and Innovation Fund, financed under the ADVANCED\_24 funding scheme.

\end{document}